\documentclass[12pt]{article}

\usepackage{latexsym,amssymb,amsmath,amsthm,epsfig,amsfonts,graphicx,epstopdf,subfig,color,float,tikz-cd,float}
\usepackage[bottom = 4cm]{geometry}

\DeclareMathOperator*{\Argmax}{Arg\,max}

\def\bse{\begin{eqnarray*}}
\def\ese{\end{eqnarray*}}
\def\be{\begin{eqnarray}}
\def\ee{\end{eqnarray}}
\def\bq{\begin{equation}}
\def\eq{\end{equation}}
\def\bse{\begin{eqnarray*}}
\def\ese{\end{eqnarray*}}

\newtheorem{definition}{Definition}
\newtheorem{theorem}[definition]{Theorem}
\newtheorem{lemma}[definition]{Lemma}

\newtheorem{example}[definition]{Example}

\newtheorem{proposition}[definition]{Proposition}
\newtheorem{conjecture}[definition]{Conjecture}

\begin{document}

\title{On the conjecture by Demyanov-Ryabova in converting finite exhausters}
\date{}
\author{Tian Sang \thanks{Discipline of Mathematical Sciences, School of Science, RMIT University, Australia.\newline
Email: s3556268@student.rmit.edu.au}}

\maketitle

\begin{abstract}
In this paper, we prove the conjecture of Demyanov and Ryabova on the length of cycles in converting exhausters in an affinely independent setting and obtain a combinatorial reformulation of the conjecture.

Given a finite collection of polyhedra, we can obtain its ``dual" collection by forming another collection of polyhedra, which are obtained as the convex hull of all support faces of all polyhedra for a given direction in space. If we keep applying this process, we will eventually cycle due to the finiteness of the problem. Demyanov and Ryabova claim that this cycle will eventually reach a length of at most two.

We prove that the conjecture is true in the special case, that is, when we have affinely independent number of vertices in the given space. We also obtain an equivalent combinatorial reformulation for the problem, which should advance insight for the future work on this problem.
\end{abstract}

\section*{Introduction}

Exhausters are multiset objects that generalise the subdifferential of a convex function. Such constructions are popular in applied optimisation as they allow for exact calculus rules and easy conversion from `upper' to `lower' characterisations of the directional derivative. Exhausters were introduced by Demyanov in \cite{Demyanov-Orig} and attracted a noticeable following in the optimisation community \cite{AbbasovDemyanov,Abbasov,DemyanovRoshchina,GVTM,KALM,KMURGJKYGTDSM,Murzabekova,Uderzo}. Exhausters and other constructive generalisations of the convex subdifferential such as quasi- and codifferentials allow for straightforward generalisation of Minkowski duality that is not available for other classic constructions \cite{Ioffe,Kruger}. Neither the essentially primal graphical derivatives \cite{RockafellarWets} nor dual coderivative objects \cite{Mordukhovich} allow for well-defined dual characterisations. The exhauster approach is not without drawbacks: such constructions inherently lack uniqueness, and whilst some works are dedicated to finding minimal objects \cite{Roshchina}, it is shown that minimal exhausters do not exist in some cases \cite{Nonunique}. The conjecture that we are studying in this paper is in a similar vein: we want to establish the uniqueness of a dual characterisation of a function by establishing a steady 2-cycle in the relevant dynamical system defined by the conversion operator.

Constructive nonsmooth subdifferentials are well suited for practical applications, especially in finite dimensional continuous problems with minimax structure of the objective function, and have been utilised successfully both in applied problems such as data classification (see an overview \cite{Bagirov}) and in theoretical problems coming from other fields, such as spline approximation \cite{Splines}.

Given a positively homogeneous function $h:\mathbb{R}^n \to \mathbb{R}$, its upper exhauster $E^*$ is a family of closed convex sets such that $h$ has an exact representation
$$h(x) = \inf_{C\in E^*}\sup_{v\in C}\langle v, x\rangle,$$
so that $h$ is the infimum over a family of sublinear functions. An upper exhauster $E^*$ is the collection of subdifferentials of these functions. The lower exhauster $E_* h$ is defined symmetrically as a supremum over a family of superlinear functions. Exhausters constructed for first order homogeneous approximations of nonsmooth functions (such as Dini and Hadamard directional derivatives) provide sharp optimality conditions, moreover, exhausters enjoy exact calculus rules which makes them an attractive tool for applications.

An upper exhauster can be converted into a lower one and vice versa using a convertor operator introduced in \cite{DemyanovNewTools}. Upper exhauster is a more convenient tool for checking the conditions for a minimum (and vice versa, lower exhauster is better suited for maximum); conversion is also necessary for the application of some calculus rules.


When the positively homogeneous function $h$ is piecewise linear, it can be represented as a minimum over a finite set of piecewise linear convex functions described by the related finite family of polyhedral subdifferentials. The exhauster conversion operator allows to obtain  symmetric local representation as the maximum over a family of polyhedral concave functions, and vice versa, where the families of sets remain finite and polyhedral. The Demyanov-Ryabova conjecture states that if this conversion operator is applied to a family of polyhedral sets sufficiently many times, the process will stabilise with a 2-cycle. Here we focus on a geometric formulation of this conjecture that does not rely on nonsmooth analysis background.

In this paper, we will first define the conversion operator and explain the statement of the conjecture. Then we will prove this conjecture in the special case. We will restrict the conjecture to the case with $n+1$ affinely independent vertices in an $n$ dimensional space, and then prove it is always true. In the final section of the paper, we will reformulate this geometric problem into an algebraic problem by considering the orderings on the vertex set and forming a simplified map. Then we will show the algebraic formulation and the geometric problem are equivalent.\\[+10pt]

\section*{Preliminaries}

Given a polyhedron and a direction, we can define the \emph{supporting face} of this polyhedron as the set of points which project the furthest along the given direction. \\[+5pt]
\begin{definition}\label{def:support-face}
Let $d \in \mathbb{S}^{n-1}$ be a direction and $P$ a polyhedron, we define $P_d$ be \emph{the supporting face of $P$ at direction $d$} (see Fig. \ref{fig:SupportFace}). That is,
$$P_d := \Argmax_{x \in P} \langle x, d \rangle  = \Argmax_{x \in P} (d^{T}x).$$
\end{definition}
Note that always $(P_d)_d = P_d$.
\begin{figure}[H]
    \begin{center}
	   \includegraphics[scale = 0.8]{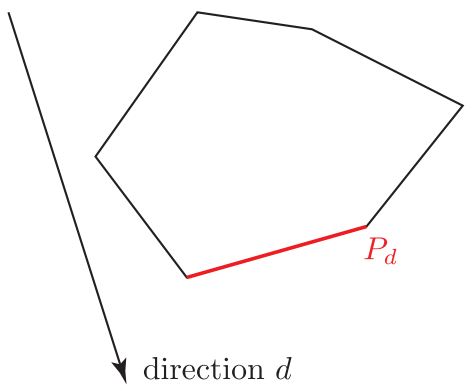}\\[+30pt]
    \end{center}
    \caption{Supporting face of a polyhedron}\label{fig:SupportFace}
\end{figure}
Let $\Omega$ be a finite family of polyhedra. Fix a direction $d$ and take the convex hull of all support faces in this directions for all polyhedra in the family. The new collection of sets generated in this fashion from all  directions $d\in \mathbb{S}^{n-1}$ is the output of Demyanov convertor.

\begin{definition}
	For a direction $d$ and a collection of polyhedra $\Omega$, we can define
	$$\Omega (d) : = Conv\big( \{P_d \ | \ P \in \Omega\} \big).$$
\end{definition}
\begin{definition}
	Define the \emph{transformation $F$} to be
    \begin{equation}\label{F}
	   F(\Omega) := \{\Omega(d) \ | \ d \in \mathbb{S}^{n-1}\}.
    \end{equation}
\end{definition}
Let $\Omega_{i+1} = F(\Omega_i)$ for $i = 0,1,2,...$. 
We are now ready to state the Demyanov-Ryabova conjecture.
\begin{conjecture}\label{conj}
	There exists $N \in \mathbb{N}$ such that if $ n > N$, then $\Omega_{n+2} = \Omega_{n}$.
\end{conjecture}

In the sequel we will use the following two reformulations of Conjecture~\ref{conj}.
\begin{lemma}\label{lem:conj-equiv} Let $\Omega_0$ be a finite family of polyhedral sets in $\mathbb{R}^n$. Conjecture~\ref{conj} is equivalent to each of the following statements.
\begin{itemize}
    \item[(1)] There exist an $N \in \mathbb{N}$ such that if $n > N$, then any polyhedron $P$ satisfies $P \in \Omega_n \Leftrightarrow P \in \Omega_{n+2}$.
    \item[(2)] Given a polyhedron $P$. Then there exist $N \in \mathbb{N}$ such that if $n > N$, then $P \in \Omega_n \Leftrightarrow P \in \Omega_{n+2}$.
\end{itemize}
\end{lemma}
\begin{proof}
It is evident that statement (1) is equivalent to Conjecture~\ref{conj}, also (1) is stronger than (2). Statement (2) yields (1) due to the finiteness of our setting: there are finitely many polyhedra that can be formed on a finite set of vertices, hence, we only need to check (2) for finitely many polyhedra, hence there exists $N$ for which (2) holds for all $P$ in this finite collection, which we can then substitute in (1).
\end{proof}

Observe that $Conv(F(X)) = Conv(X)$ for any set of polyhedra $X$. So $Conv(\Omega_i)$ is constant. We let $C = Conv(\Omega_0) = Conv(\Omega_i)$ for all $i\in \mathbb{N}$, and by $C_d$ we denote the  supporting face of $C$ in direction $d$ in alignment with the notation of Definition~\ref{def:support-face}.

\begin{example}
For the example shown in Figure~\ref{fig:five-sets}, $C$ is the convex hull of 5 convex sets (i.e. a single vertex, two line segments, a triangle, and a rectangle) in $\mathbb{R}^2$. Every edge and vertex in $C$ is a supporting face for some direction $d$.
\begin{figure}[H]
	\begin{center}
		\includegraphics[scale = 0.8]{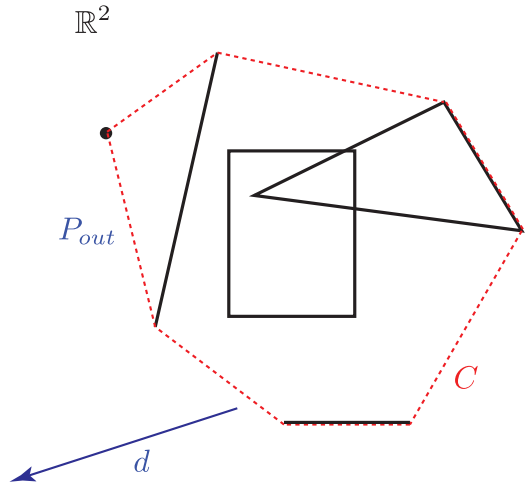}\\[+30pt]
	\end{center}
    \caption{Convex hull $C$ of 5 sets}\label{fig:Conv5Sets}\label{fig:five-sets}
\end{figure}
\end{example}

\section*{Affinely independent case}

The main goal of this section is to prove that Conjecture \ref{conj} is true for affinely independent case, i.e. when all vertices in of the polyhedra in our family form an affinely independent set. 
We begin with several technical claims and finish with the proof of the main result in  Theorem~\ref{thm:main-affine}.

\begin{lemma}\label{Cd_in_omega}
	If $C_d \in \Omega_n$, then $C_d \in \Omega_{n+2}$.
\end{lemma}
\begin{proof}
Let $C_d \in \Omega_n$, where $d$ is a direction.\\[+5pt]
Then we know that $(C_d)_d = C_d$. \\[+5pt]
By definition of $\Omega(d)$, we have
$$\Omega_n(d) = Conv\big( \{P_d \ | \ P \in \Omega_n\} \big)$$
By assumption $C_d \in \Omega_{n}$, we have
$$C_d = (C_d)_{d} \subseteq \Omega_n(d)$$
Also $\Omega_n(d) \subseteq C$, but $(C_d)_d = C_d$, therefore
$$C_d = \Omega_n(d)_d$$
Since $\Omega_n(d) \in \Omega_{n+1}$, we have one inclusion
$$\Omega_{n+1}(d) \supseteq \Omega_n(d)_d = C_d.$$\\[+5pt]
Now we will show the other inclusion $\Omega_{n+1}(d) \subseteq C_d$.\\[+5pt]
This is equivalent of showing: If $P \in \Omega_{n+1}$, then $P_d \subseteq C_d$.\\[+5pt]
For a given $d \in \mathbb{S}^{n-1}$, there exists $P \in \Omega_{n+1}$ such that $P = \Omega_n(d)$.\\[+5pt]
Therefore,
$$P = \Omega_n(d) = Conv\big( \{ P_d \ | \ P \in \Omega_n \} \big) \supseteq P_d$$
Since $P_d \subseteq (P \cap C_d)$, we get $P \cap C_d \neq \emptyset$.\\[+5pt]
Hence,
$$P_d = (P \cap C_d) \subseteq C_d \ \ \ \text{ (by definition of the supporting face) }$$\\[+5pt]
Now we have both inclusions, we have shown that $C_d = \Omega_{n+1}(d)$ for any given $d$.\\[+5pt]
Given the fact $\Omega_{n+1}(d) \in \Omega_{n+2}$, we have the result $C_d \in \Omega_{n+2}$.
\end{proof}

\begin{proposition}\label{iff_p_out_in_Omega}
	There exists $N \in \mathbb{N}$ such that if $n > N$, then $C_d \in \Omega_n$ if and only if $C_d \in \Omega_{n+2}$.
\end{proposition}
\begin{proof}
By the Lemma \ref{Cd_in_omega}, we have either:
\begin{itemize}
	\item[(a)] $C_d \notin \Omega_{0}$, $C_d \notin \Omega_{2}$, ... , $C_d \in \Omega_{2k+2}$, $C_d \in \Omega_{2k+4}$, ... for some $k \in \mathbb{N}$.
    \item[(b)] $C_d \notin \Omega_{2k}$ for all $k \in \mathbb{N}$, and $C_d \in \Omega_{2k+1}, C_d \in \Omega_{2k+3}, ...$ for some $k$.
	\item[(c)] $C_d \in \Omega_{2k}$ for all $k \in \mathbb{N}$.
\end{itemize}
In any of these cases, there exist $N_1 \in \mathbb{N}_{>0}$ such that if $n > N_1$, and $2$ divides $n$, then, $$C_d \in \Omega_{n} \Leftrightarrow C_d \in \Omega_{n+2}.$$
Similarly, there exist $N_2 \in \mathbb{N}_{>0}$ such that if $n > N_2$ and $2$ does not divide $n$, then,
$$C_d \in \Omega_{n} \Leftrightarrow C_d \in \Omega_{n+2}.$$
Therefore, we can set $N = \max\{N_1, N_2\}$, which proves the proposition.
\end{proof}
Recall the definition of simplex below:
\begin{definition}
    Let $k+1$ points $v_0, v_1, ... , v_k \in \mathbb{R}^n$ be affinely independent. Then, the simplex determined by this set of points is:
    $$C = \{\lambda_0 v_0 + \cdots + \lambda_k v_k \ | \ \lambda_i \geq 0, \ \sum_{i = 0}^k \lambda_i = 1\}.$$
\end{definition}
In other words, a simplex is the generalisation of a tetrahedral region of spaces to an arbitrary dimension. A \emph{k-simplex} is a $k$-dimensional polytope that is a convex hull of its $k+1$ vertices. Observe that every face of a simplex (\emph{sub-simplex}) is still a simplex in its lower dimensional space.\\[+10pt]

\begin{theorem}\label{thm:main-affine}
    If $C$ is a simplex, and each $P \in \Omega$ is a sub-simplex of $C$, then the conjecture is true.
\end{theorem}
\begin{proof}
By induction, we can show that every $P \in \Omega_i$ and every $P_d$ for any direction $d$ is a sub-simplex of $C$.\\[+3pt]
For any sub-simplex $P$ of $C$, with $P \neq C$, $P$ is a supporting face of $C$. Therefore, there exist $N_P$ such that for $n \geq N_P$, we have
$$P \in \Omega_n \Leftrightarrow P \in \Omega_{n+2} \ \ \ \ \text{(by Proposition \ref{iff_p_out_in_Omega})}.$$
Let
$$N := \max\{N_P \ | \ P \text{ is a sub-simplex of }C, P \neq C\} + 2.$$
Then $\Omega_{N} = \Omega_{N+2}$ since $N-2 \geq N_P$ for any sub-simplex $P$ of $C$ satisfying $P \neq C$.\\[+5pt]
Therefore, $P \in \Omega_{N-2} \Leftrightarrow P \in \Omega_{N}$, $P \in \Omega_{N-1} \Leftrightarrow P \in \Omega_{N+1}$ and $P \in \Omega_N \Leftrightarrow P \in \Omega_{N+2}$.\\[+15pt]
We have shown the statement for proper faces, now we will show it is true for $C$, which is
$$C \in \Omega_N \Leftrightarrow C \in \Omega_{N+2}.$$
Suppose $C \in \Omega_N$ and $C \notin \Omega_{N+2}$, then $C = \Omega_{N-1}(d)$ for some direction $d$.\\[+5pt]
But $\Omega_{N+1}(d) \neq C$ since $C \notin \Omega_{N+2}$ by assumption.\\[+5pt]
Let $a \in C$ be a vertex with $a \notin \Omega_{N+1}(d)$, then $a \in \Omega_{N-1}(d)$ implies $a \in P_d$ for some $P \in \Omega_{N-1}$.\\[+5pt]
Since $a \notin \Omega_{N+1}(d)$, then $P \notin \Omega_{N+1}$.\\[+5pt]
Since $P \in \Omega_{N-1}$ and $P \notin \Omega_{N+1}$, we have
$$P = C \ \ \ \text{(Otherwise we get }P \in \Omega_{N-1} \Leftrightarrow P \in \Omega_{N+1}).$$
Hence $C \in \Omega_{N-1}$ and $C \notin \Omega_{N+1}$.\\[+5pt]
Therefore $C \in \Omega_N$ and $C \notin \Omega_{N+2}$ implies $C \in \Omega_{N-1}$ and $C \notin \Omega_{N+1}$.\\[+10pt]
Similarly, $C \in \Omega_{N-1}$ and $C \notin \Omega_{N+1}$ implies $C \in \Omega_{N-2}$ and $C \notin \Omega_{N}$.\\[+5pt]
Thus $C \in \Omega_{N}$ and $C \notin \Omega_{N+2}$ implies $C \notin \Omega_{N}$, a contradiction. \\[+10pt]
Therefore,
$$C \in \Omega_{N} \Rightarrow C \in \Omega_{N+2}.$$
By similar argument,
$$C \in \Omega_{N+2} \Rightarrow C \in \Omega_{N}.$$
Since $C \in \Omega_{N+2}$ and $C \notin \Omega_N$ implies $C \in \Omega_{N+1}$ and $C \notin \Omega_{N-1}$, this implies $C \in \Omega_{N}$ and $C \notin \Omega_{N-2}$, which contradicts $C \notin \Omega_{N}$.\\[+5pt]
Therefore $C \in \Omega_N \Leftrightarrow C \in \Omega_{N+2}$, which implies $\Omega_{N} = \Omega_{N+2}$. \end{proof}

\section*{Algebraic reformulation of the conjecture using orderings on vertex set}
We can formulate this geometric problem into an algebraic problem by ordering the vertex set.\\[+5pt]
Firstly, we label all the vertices of the polyhedra in $\Omega_0$, the order doesn't matter.\\[+5pt]
After that, we pick a direction $d$, then we can ``encode" $d$ by writing the vertex set in order of furthest to closest along the $d$ direction.\\[+5pt]
We ignore the directions such that having more than one vertex are furthest along the direction. In other words, we ignore the directions perpendicular to edges of the polyhedra.\\[+5pt]
Then we know that based on the description of the transformation, every direction gives a convex hull. For each direction $d$, we compare the encoded word of the direction with the polyhedra from the previous state, then we can write down the precise vertex set of the convex hull that is created.\\[+5pt]

\begin{lemma}
    Let $n$ be the number of vertices in $\mathbb{R}^2$. If there are no more than two vertices collinear, then we have exactly $n(n-1)$ number of directions.
\end{lemma}
\begin{proof}
Let $d$ be an arbitrary direction. Then we can rotate $d$ clockwise to obtain all directions. \\[+3pt]
We can encode $d$ by writing the vertex set in order of furthest along $d$ to closest along $d$. As we rotate the direction $d$ clockwise, each pair of letters swaps exactly twice. This implies that there are $2 \times {n \choose2} = n(n-1)$ swaps in total.\\[+5pt]
Also we know that each swap gives a new ordering on the vertex set. Therefore, there are $n(n-1)$ vertex orders in total.
\end{proof}
\emph{Note}: If there are 3 or more vertices collinear, or two or more pairs of collinear vertices are parallel to each other, then the number of orders for vertex set would be less than $n(n-1)$, as some of the swaps would happen the same time as we rotate the direction around the $\mathbb{R}^2$ plane.\\[+5pt]
Therefore, the upper bound of the number of the directions is $n(n-1)$ for the general case.\\[+10pt]

\begin{example}\label{alg_exp}
    Consider the following example in $\mathbb{R}^2$ which start with a set contains a line segment and a single vertex in $\mathbb{R}^2$. Then we name the three vertices as $A, B$, and $C$.
    \begin{figure}[H]
    \begin{center}
        \includegraphics[scale = 0.8]{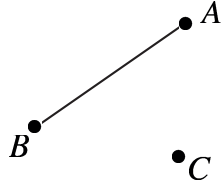}
    \end{center}
    \caption{Two sets with three vertices in $\mathbb{R}^2$}\label{fig:Sets3Vertices}
    \end{figure}
    Then we pick a direction $d$, and encode the direction use a word in terms of vertices in order of furthest along the $d$ to closest along the direction.
    \begin{figure}[H]
    \begin{center}
        \includegraphics[scale = 0.8]{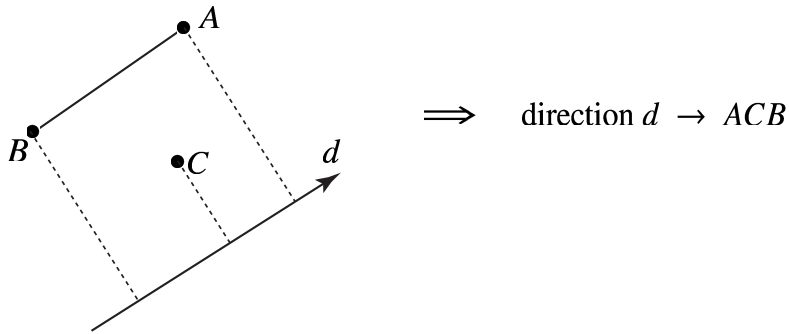}
    \end{center}
    \caption{Encode the direction $d$}\label{fig:EncodeDirectd}
    \end{figure}
    There will be 6 directions in total, which are:
    $$ACB, \ ABC, \ BCA, \ BAC, \ CAB, \ CBA$$
    The starting set of polyhedra is $\Omega_0 = \{AB, C\}$.\\[+5pt]
    Now suppose we want to know what is the convex hull created by direction $ACB$. We move in order along each letter in the direction $ACB$, and check with each polyhedron in $\Omega_0$. ``$A$" is in polyhedron ``$AB$", then we stop and move onto the next polyhedron in $\Omega_0$. ``$A$" is not in polyhedron ``$C$", so we move to the next letter in the direction, which is ``$C$", ``$C$" is in polyhedron ``$C$". We can stop now as we have exhausted polyhedra in $\Omega_0$. We can conclude that the convex hull created by direction $ACB$ is $AC$.\\[+5pt]
    Similarly, we can do the same algorithm for all 6 directions, so we end up 6 polyhedra, which are:
    $$\{AC, \ AC, \ BC, \ BC, \ AC, \ BC\}$$
    Then delete the repeated elements to obtain:
    $$\Omega_1' = \{AC, \ BC\}$$
    Then we apply these 6 directions to the set $\Omega_1$ again to obtain another 6 polyhedra, which are:
    $$\{AC, \ AB, \ CB, \ AB, \ C, \ C\}$$
    Therefore we have:
    $$\Omega_2' = \{AC, \ AB, \ CB, \ C\}$$
    Then keep applying the same procedure, we get
    $$\Omega_3' = \{AC, \ ABC, \ CB\}$$
    $$\Omega_4' = \{AC, \ AB, \ CB, \ AB, \ C\} = \{AC, \ AB, \ CB, \ C\}$$
    So we have reached to a cycle of length 2.\\
\end{example}
Now we can consider the following:\\[+5pt]
$F:$ Proper transformation from the original conjecture, as defined by Equation \ref{F}.\\[+5pt]
$F':$ Pseudo transformation which ignores the directions that give whole edges.
We denote these restricted directions as $\widehat{\mathbb{S}^{n-1}}$, which is a subset of $\mathbb{S}^{n-1}$, and we denoted the corresponding images as follows:
\begin{itemize}
    \item[] $\Omega_{i+1} = F(\Omega_i)$
    \item[] $\Omega_{i+1}^{'} = F'(\Omega_i)$\\[+5pt]
\end{itemize}

\subsection*{Abstract Algebraic Formulation}
Let $V$ be a finite set, and let $\tau = \{d_j\}_{j \in \{1,..., n\}}$ be a set of orderings of the set $V$.\\[+3pt]
Let $\mathcal{P}(V)$ be the power set of set $V$. We define the function $G_{\tau} :\mathcal{P}(\mathcal{P}(V)) \longrightarrow \mathcal{P}(\mathcal{P}(V))$ as the following:\\[+5pt]
For each $j \in \{1,..., n\}$, let $\widehat{d_j}$ be the maximality function given by the ordering $d_j$, that is, $\widehat{d_j} = \max\{V\}$ given $d_j$, which means $\widehat{d_j}$ depends on $\tau$. \\[+5pt]
Define $D_j: \mathcal{P}(\mathcal{P}(V)) \longrightarrow \mathcal{P}(V)$ by,
$$D_j(X) := \{\widehat{d_j}(S) \ | \ S \in X\}$$
Where $X \in \mathcal{P}(\mathcal{P}(V))$ is a collection of subsets of $V$.\\[+5pt]
Finally, we define $G_{\tau} :\mathcal{P}(\mathcal{P}(V)) \longrightarrow \mathcal{P}(\mathcal{P}(V))$ by,
$$G_{\tau}(X) = \{D_j(X) \ | \ j \in \{1,... , n\}\}$$\\[+3pt]
\begin{example}
     Given the previous Example \ref{alg_exp}, we have the following corresponding algebraic structure based on out abstract algebraic formulation above.
     \begin{itemize}
        \item[$\bullet$] $V = \{A, B, C\}$
        \item[$\bullet$] $\tau = \{ACB, ABC, BCA, BAC, CAB, CBA\}$
        \item[$\bullet$] $\mathcal{P}(V) = \big\{A, B, C, \{A,B\}, \{A,C\}, \{B,C\}, \{A,B,C\} \big\}$
        \item[$\bullet$] The maximality function $\widehat{d_j}$ is equivalent to obtaining the supporting face given the direction $d_j$.
        \item[$\bullet$] The function $D_j$ gives the convex hull of all supporting faces for a given direction $d_j$.
        \item[$\bullet$] The function $G_{\tau}$ outputs the $\Omega$ set.
     \end{itemize}
     For example, given $d_1 = ACB$ and $X_0 = \Omega_0 = \big\{ \{A,B\}, C \big\} \in \mathcal{P}(V)$, we get:
     \begin{itemize}
        \item[$\bullet$] $\widehat{d_1}(AB) = A$
        \item[$\bullet$] $\widehat{d_1}(C) = C$
     \end{itemize}
     Then we can compute the convex hull,
     $$D_1(X) = \{A,C\}$$
     Therefore, given $X_0$, we can compute $X_1$
     $$X_1 = G_{\tau}(\big\{ \{A,B\}, C \big\}) = \big\{ \{A,C\}, \{B,C\}   \big\}$$
     We can then continue the process to obtain $X_2, X_3, X_4, ...$\\[+5pt]
\end{example}
The following names will be helpful on giving the equivalent algebraic version of the conjecture:
\begin{itemize}
    \item[$\bullet$] We call the function $G_{\tau}$ an \emph{oscillator} if it has the following property: For any $X_0 \in \mathcal{P}(\mathcal{P}(V))$, the sequence $X_0, X_1, ... $ defined by $X_{i+1} = G_{\tau}(X_i)$ eventually cycles with period at most 2.
    \item[$\bullet$] We call the tuple $(V, \tau)$ \emph{geometric} if $V$ is a set of vertices, and $\tau$ is given by directions $d \in \widehat{\mathbb{S}^{n-1}}$.
    \item[$\bullet$] We call $(V, \tau)$ \emph{finite} if $V$ is finite.
\end{itemize}
\emph{Equivalent conjecture}: For every finite geometric pair $(V, \tau)$, the function $G_{\tau}$ is an oscillator.\\[+5pt]

\begin{lemma}\label{Pd_in_Conv}
    Let $P_1, P_2, ... , P_k$ be polyhedra. If $P = Conv(\{P_1, P_2, ..., P_k \})$ and $d$ is the direction, then $P_d \subseteq Conv(\{(P_1)_d, (P_2)_d, ... , (P_k)_d\})$.
\end{lemma}
Intuitively, consider the diagram below:
\begin{figure}[H]
    \begin{center}
        \includegraphics[scale = 0.8]{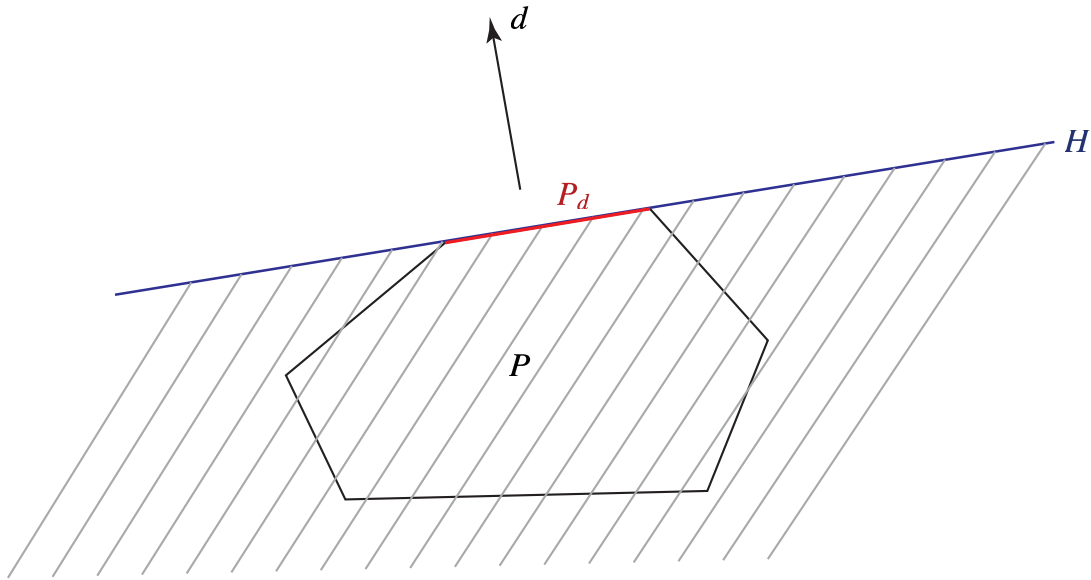}
    \end{center}
    \caption{Hyperplane $H$ contains $P_d$}\label{fig:HypContainPd}
\end{figure}
Suppose $P$ is the polyhedron from the lemma, $d$ is a direction, and $H$ is the hyperplane contains $P_d$ and orthogonal to the direction $d$. Then given a point $p \in P_d$, we can write $p$ as a convex combination, that is $p = \sum_i^k \lambda_i x_i$, with $\sum_i^k \lambda_i = 1, \lambda_i \geq 0$, and $x_i \in \bigcup_i^k P_i$.\\[+5pt]
Then we know all polyhedra $P_i$ must be below the hyperplane $H$, therefore all $x_i$ must be in the shaded area. The point $p$ is obtained by averaging the points $x_i$, therefore, all $x_i$ must be contained in the hyperplane $H$. Indeed, if there $x_j$ is below the hyperplane $H$, there needs to be another point $x_j^{'}$ above the $H$, which contradicts the fact all $x_i$ must be contained in shaded area.\\[+5pt]
\begin{proof}
Let $d \in \mathbb{S}^{n-1}$ be the linear function. By definition, we have
$$P_d := \Argmax_{x \in P} (d^{T} x) = \{x \in P \ | \ d^{T} x = \max_{x \in P}(d^{T}x)\}$$
Set $\displaystyle \alpha = \max_{x \in P}(d^{T}x)$.\\[+5pt]
Let $p \in P_d$ be an arbitrary point on $P_d$. Then we can write
$$p = \sum_i^{k} \lambda_i x_i$$
where $x_i \in P_i \subseteq P$.\\[+5pt]
Then we have
\begin{align*}
    \alpha & = d^{T} p \ \ (\text{by definition of }P_d)\\[+3pt]
            &　= d^{T} \sum_i^{k} \lambda_i x_i \\[+3pt]
            & = \sum_i^{k} \lambda_i (d^{T} x_i) \ \ (\text{since }d \text{ is linear})\\[+3pt]
            & \leq \sum_i^{k} \lambda_i \alpha \\[+3pt]
            & = \alpha \ \ (\text{since } \sum_i^k \lambda_i = 1)
\end{align*}
Therefore, the equality holds, so we have
$$\sum_i^{k} \lambda_i (d^{T} x_i) = \sum_i^{k} \lambda_i \alpha$$
which implies
$$\lambda_i (d^{T} x_i) = \lambda_i \alpha, \ \ \text{for each } i$$
which means for $\lambda_i > 0$, we get
$$d^{T}x_i = \alpha = \max_{x \in P}(d^{T}x)$$
In order to maximise $\sum_i^{k} \lambda_i (d^{T} x_i)$, we need to maximise each $d^{T} x_i$ with $\lambda_i > 0$.\\[+5pt]
By definition, $(P_i)_d$ is the subset of polyhedron $P_i$ such that $d^{T} x_i$ is maximal for $x_i \in P_i$.
$$(P_i)_d := \{x_i \in P_i \ | \ d^{T}x_i = \max_{x_i \in P_1}(d^{T}x_i)\}$$
Since $x_i \in P_i$ and $d^{T}x_i$ is maximal, then $x_i \in (P_i)_d$.\\[+5pt]
Therefore, we conclude that $P_d \subseteq Conv(\{(P_1)_d, (P_2)_d, ... , (P_k)_d\})$.\\[+5pt]
\end{proof}


\begin{lemma}\label{fac_vert_neig_exp}
    Let $P$ be a polytope, $g \in \mathbb{S}^{n-1}$ exposes a face $F$ of $P$.\\[+3pt]
    Then $g$ has a neighbourhood $N_g$ such that any $g' \in N_g \cap \mathbb{S}^{n-1}$ exposes a vertex in $F$.
\end{lemma}
\begin{proof}
We will prove the result by contradiction.\\[+3pt]
Suppose to the contrary, then there exists a sequence of restricted directions $\{d_j^{'}\}_{j = 1}^{\infty}$ converging to $d$ such that each $d_j^{'}$ exposes a vertex not in $F$.\\[+3pt]
Since $P$ is finite, without loss of generality, we can assume that each $d_j^{'}$ exposes the same vertex $v \notin F$.\\[+3pt]
Let $u \in F$ be an arbitrary point. Since $v \notin F$, we have
$$\langle d,u \rangle > \langle d,v \rangle .$$
But for each $j$, $d_j^{'}$ exposes $v$, therefore, we have
$$\langle d_j^{'},u \rangle < \langle d_j^{'},v \rangle .$$
Given $\{d_j^{'}\} \longrightarrow d$ as $j \rightarrow \infty$, we have
$$\langle d_j^{'}, \cdot \rangle \longrightarrow \langle d, \cdot \rangle$$
Therefore, we have
$$\langle d,u \rangle \leq \langle d,v \rangle .$$
A contradiction.
\end{proof}
\begin{theorem}\label{P=ConvP'}
    If $P \in \Omega_i$, then there exist $P_1', ..., P_k' \in \Omega_i^{'}$ such that $P = Conv(\{P_1', ... , P_k'\})$.
\end{theorem}
To prove Theorem \ref{P=ConvP'}, we need the two following lemmas (see \cite{Brondsted} for reference).\\[+5pt]
\begin{lemma}\label{vert_neig}
    If $d \in \mathbb{S}^{n-1}$ exposes a vertex $v$ from a polyhedron $P$. That is,
    $$\langle x-v, g \rangle < 0 \ \ \forall x \in P \setminus \{v\}.$$
    Then there is a neighbourhood of $d$ in which every direction also exposes $v$.
\end{lemma}
\begin{lemma}\label{fac_vert_exp}
    Let $P$ be a polyhedron, $d \in \mathbb{S}^{n-1}$ exposes a face of $P$, and $v$ be an extreme point of the face exposed by $d$. Then for all $\epsilon > 0$, there exists $d'$ such that $||d - d'|| \leq \epsilon$ and $d'$ exposes $v$.\\[+5pt]
\end{lemma}
\begin{proof} (Theorem \ref{P=ConvP'})\\[+5pt]
Let $P \in \Omega_i$, and let $\Omega_{i-1} = \{P_1, P_2, ..., P_n\}$.\\[+3pt]
Let $g$ be the direction such that $P = \Omega_{i-1}(g) = Conv(\{(P_1)_g, (P_2)_g, ... , (P_n)_g\})$.\\[+3pt]
Now we want to show two things:
\begin{itemize}
    \item[(1)] $g$ has a neighbourhood $N_g$ such that for any $g' \in N_g \cap \mathbb{S}^{n-1}$, $\Omega_{i-1}(g') \subseteq P$.
    \item[(2)] For any vertex $v \in P$, there is a $g' \in N_g \cap \mathbb{S}^{n-1}$ such that $v \in \Omega_{i-1}(g')$.
\end{itemize}
Part (1): Given $\Omega_i = \{P_1, ..., P_n\}$, by Lemma \ref{fac_vert_neig_exp}, we can construct $N_j$ of $g$ for each polygon $P \in \Omega_{i-1}$, so that any $g' \in N_j \cap \mathbb{S}^{n-1}$ exposes a vertex of $P_j$ in $(P_j)_g$.\\[+3pt]
Let $N_g = \cap_{j = 1}^{n} N_j$.\\[+3pt]
Since $g' \in N_g \cap \mathbb{S}^{n-1}$ is a restricted direction, and the restricted directions only expose vertices. Then for each $j$, we have $g' \in N_j \cap \mathbb{S}^{n-1}$. Hence $g'$ exposes a vertex of $P_j$ in $(P_j)_g \subseteq P$. \\[+3pt]
So we have
$$(P_j)_{g'} \subseteq P.$$
Therefore,
$$\Omega_{i-1}(g') = Conv(\{(P_j)_{g'}\}) \subseteq P.$$\\[+3pt]
Part(2): Let $v \in P$ be a vertex, then we know $v \in (P_j)_g$ for some $j$.\\[+3pt]
Now, by Lemma \ref{fac_vert_exp}, there is a $g' \in N_g \cap \mathbb{S}^{n-1}$ which exposes $v \in P_j$.\\[+3pt]
Therefore,
$$v \in \Omega_{i-1}(g').$$\\[+3pt]
Use Part (2) we have shown above, let $g_1^{'}, g_2^{'}, g_3^{'}, ... , g_g^{'} \in N_g$ be restricted directions such that every vertex $v \in P$ is contained in some $\Omega_{i-1}(g_j^{'})$.\\[+3pt]
Let $P_j^{'} = \Omega_{i-1}(g_j^{'})$ for some $j$. Then every vertex $v \in P$ is contained in $Conv(\{P_1^{'}, ... , P_k^{'}\})$.\\[+3pt]
Therefore, we get one inclusion:
$$P \subseteq Conv(\{P_1^{'}, ... , P_k^{'}\}).$$\\[+3pt]
Now by Part (1), since each $g_j^{'} \in N_g$, we have each $P_j^{'} = \Omega_{i-1}(g_j^{'}) \subseteq P$.\\[+3pt]
Hence,
$$Conv(\{P_1^{'}, ... , P_k^{'}\}) \subseteq P .$$\\[+3pt]
Therefore, we have the equality:
$$P = Conv(\{P_1^{'}, ... , P_k^{'}\})$$\\[+3pt]
\end{proof}

Recall that we denote our two transformations by:
\begin{itemize}
    \item[] $\Omega_{i+1} = F(\Omega_i)$ and
    \item[] $\Omega_{i+1}^{'} = F'(\Omega_i)$
\end{itemize}
\begin{theorem}\label{F'_maps}\
    Given the two transformations we had above, there exists two following maps:
    \begin{itemize}
        \item[$\bullet$] $\Omega_{i+1} = F(\Omega_i^{'})$
        \item[$\bullet$] $\Omega_{i+1}^{'} = F'(\Omega_i^{'})$
    \end{itemize}
\end{theorem}
That is:
\begin{center}
    \includegraphics[scale = 0.25]{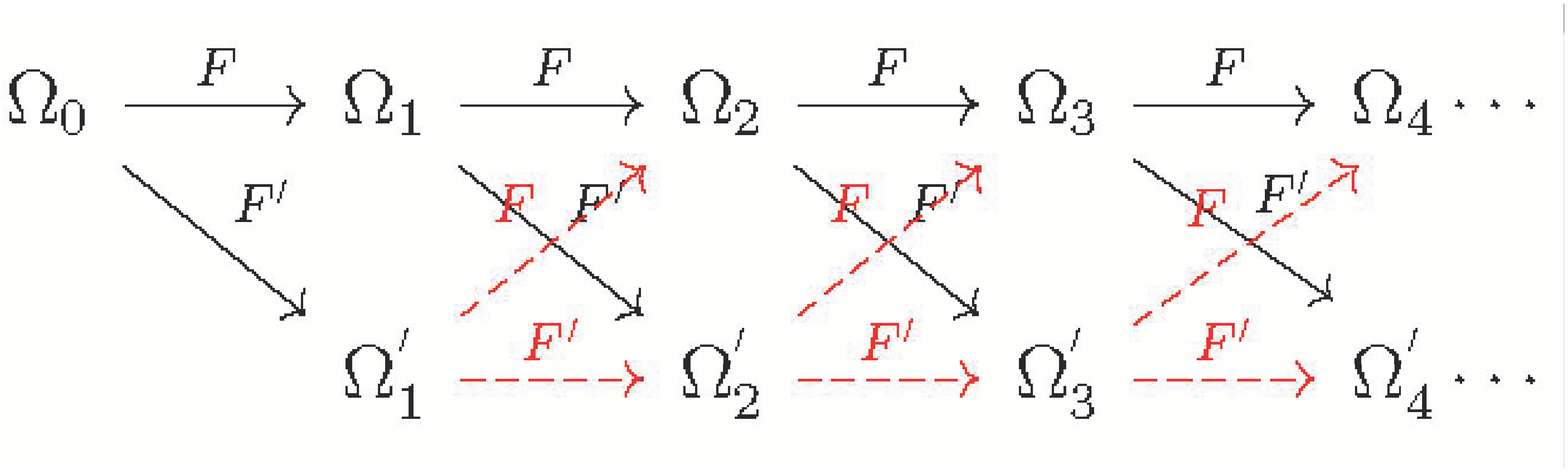}
\end{center}
\begin{proof}
This is equivalent to showing:
    \begin{itemize}
        \item[$\bullet$] $F(\Omega_i^{'}) = F(\Omega_i) = \Omega_{i+1}$
        \item[$\bullet$] $F'(\Omega_i^{'}) = F'(\Omega_i) = \Omega_{i+1}^{'}$
    \end{itemize}
    For the first equality $F(\Omega_i^{'}) = F(\Omega_i)$, this is the same as showing:
    $$\{\Omega_i^{'}(d) \ | \ d \in \mathbb{S}^{n-1}\} = \{\Omega_i(d) \ | \ d \in \mathbb{S}^{n-1}\}$$
    It is sufficient to show that for all $d \in \mathbb{S}^{n-1}$, we have:
    $$\Omega_i^{'}(d) = \Omega_i(d)$$
    By definition, this is the same as showing:
    $$Conv(\{P_d^{'} \ | \ P' \in \Omega_i^{'}\}) = Conv(\{P_d \ | \ P \in \Omega_i\})$$
    \begin{itemize}
       \item[$\bullet$] If $P' \in \Omega_i^{'}$, then $P' \in \Omega_i$. Therefore, we have $$ \{P_d^{'} \ | \ P' \in \Omega_i^{'}\} \subseteq \{P_d \ | \ P \in \Omega_i\}.$$
           Therefore, each element of $\{P_d^{'} \ | \ P' \in \Omega_i^{'}\}$ is also an element in $\{P_d \ | \ P \in \Omega_i\}$.\\[+3pt]
           Hence, we have one inclusion
           $$Conv(\{P_d^{'} \ | \ P' \in \Omega_i^{'}\}) \subseteq Conv(\{P_d \ | \ P \in \Omega_i\}).$$
        \item[$\bullet$] Let $P \in \Omega_i$, then by Theorem \ref{P=ConvP'}, we have
        $$P = Conv(\{P_1^{'}, P_2^{'}, ... , P_k^{'}\}), \ P_j^{'} \in \Omega_i^{'}.$$
            Then by Lemma \ref{Pd_in_Conv}, we have
            $$P_d \subseteq Conv(\{(P_1^{'})_d, (P_2^{'})_d, ... , (P_k^{'})_d\}).$$
            But, $$\{(P_1^{'})_d, (P_2^{'})_d, ... , (P_k^{'})_d\} \subseteq \{P_d^{'} \ | \ P' \in \Omega_i^{'}\}$$
            Which implies
            $$Conv(\{(P_1^{'})_d, (P_2^{'})_d, ... , (P_k^{'})_d\}) \subseteq Conv(\{P_d^{'} \ | \ P' \in \Omega_i^{'}\})$$
            Hence, we have
            $$P_d \subseteq Conv(\{P_d^{'} \ | \ P' \in \Omega_i^{'}\})$$
            Therefore, we have the other inclusion
            $$Conv(\{P_d \ | \ P \in \Omega_i\}) \subseteq Conv(\{P_d^{'} \ | \ P' \in \Omega_i^{'}\}).$$
    \end{itemize}
    Therefore, we have shown
    $$Conv(\{P_d^{'} \ | \ P' \in \Omega_i^{'}\}) = Conv(\{P_d \ | \ P \in \Omega_i\}),$$
    which means
    $$\Omega_i^{'}(d) = \Omega_i(d),$$
    for all $d \in \mathbb{S}^{n-1}$\\[+5pt]
    This gives the desired result:
    $$F(\Omega_i^{'}) = F(\Omega_i).$$\\[+5pt]
    To show the second map $F'(\Omega_i^{'}) = F'(\Omega_i)$, we need to show
    $$\{\Omega_i^{'}(d') \ | \ d' \in \widehat{\mathbb{S}^{n-1}}\} = \{\Omega_i(d') \ | \ d' \in \widehat{\mathbb{S}^{n-1}}\},$$
    But, we have shown from the previous map $F(\Omega_i^{'}) = F(\Omega_i)$ that
    $$\Omega_i^{'}(d) = \Omega_i(d),$$
    for all direction $d \in \mathbb{S}^{n-1}$.\\[+3pt]
    Therefore, it must be true for all restricted directions $d' \in \widehat{\mathbb{S}^{n-1}}$.\\[+3pt]
    So we have the map
    $$F'(\Omega_i^{'}) = F'(\Omega_i).$$
\end{proof}

\begin{theorem}\label{equavalence}
    $\Omega_0^{'}, \Omega_1^{'}, \Omega_2^{'}, ... $ eventually reach to a cycle of length 2 if and only if $\Omega_0, \Omega_1, \Omega_2, ... $ eventually also reach to a cycle of length 2.
\end{theorem}
\begin{proof}
The result follows directly from the Theorem \ref{F'_maps}.\\[+5pt]
$(\Rightarrow)$ Suppose $\Omega_i' = \Omega_{i+2}'$, this implies that $F(\Omega_i^{'}) = F(\Omega_{i+2}^{'})$, which gives $\Omega_{i+1} = \Omega_{i+3}$.\\[+5pt]
$(\Leftarrow)$ Similarly, suppose $\Omega_i = \Omega_{i+2}$, this implies that $F'(\Omega_i) = F'(\Omega_{i+2})$, which gives $\Omega_{i+1}^{'} = \Omega_{i+3}^{'}$.
\end{proof}
\emph{Note}: We cannot apply the $F$ map or $F'$ map to the equation, as both maps may not be injective. Therefore we have to prove the Theorem \ref{equavalence} back track maps.\\[+5pt]

\begin{example}
    Consider the following example:
   \begin{figure}[H]
   \begin{center}
       \includegraphics[scale = 0.8]{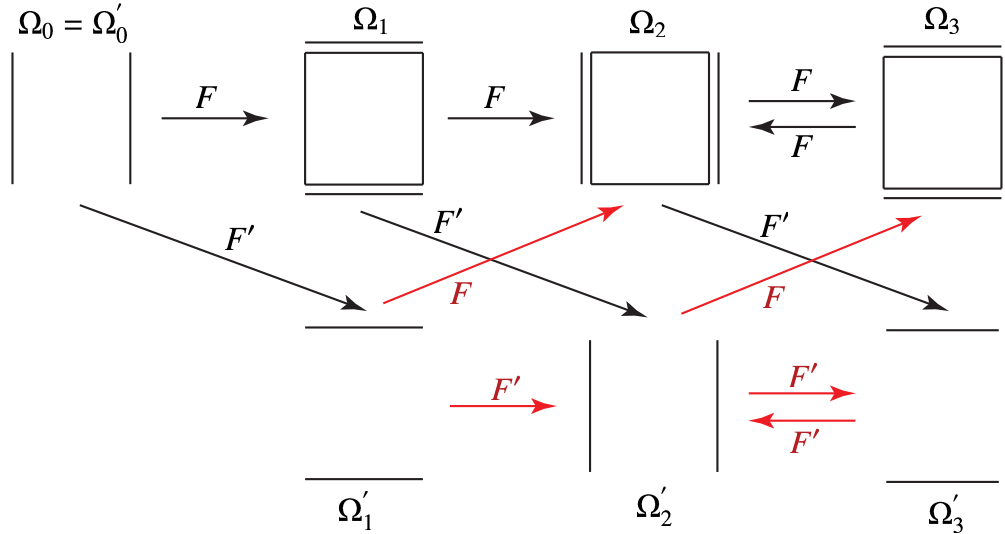}\\[+10pt]
   \end{center}
   \caption{An example on comparison between maps $F$ and $F'$}\label{fig:EgMapsFF'}
   \end{figure}
\end{example}
The importance of the abstract algebraic formulation result allows us to work on the conjecture using more general algebra. After we obtain the set of orderings on vertex set that corresponding the set of restricted directions, we are able to forget about the geometry of the sets, and proceed with the equivalent algebraic version of the conjecture, which means we may can apply many powerful algebraic and combinatorial tools on this problem.

\bibliographystyle{plain}
\bibliography{paper}

\end{document}